\documentclass[11pt,reqno]{amsart} 

\usepackage{amssymb,latexsym}
\usepackage[draft=true]{hyperref}
\usepackage{cite} 

\usepackage[height=190mm,width=130mm]{geometry} 

\theoremstyle{plain}
\newtheorem{theorem}{Theorem}
\newtheorem{lemma}{Lemma}

\theoremstyle{definition}

\theoremstyle{remark}

\numberwithin{equation}{section} 

\begin{document}
\title[Half-Inverse Problem]{A half-inverse problem for the Singular Diffusion Operator with Jump Conditions} 

\author{ Abdullah ERG\"{U}N}
\address{Cumhuriyet University\\ Vocational School of Sivas\\ Sivas.\\ 58140
  \\ Turkey}

\email{aergun@cumhuriyet.edu.tr}

\begin{abstract}
In this paper, half inverse spectral problem for diffusion operator with jump conditions  dependent on the spectral parameter and  discontinuoty coefficient is considered. The half inverse problems is studied of determining the coefficient and two potential functions of the boundary value problem its spectrum by Hocstadt- Lieberman and Yang-Zettl methods. We show that two potential functions on the whole interval and the parameters in the boundary and jump conditions can be determined from spectrum.
\end{abstract}




\maketitle

\section{ Introduction and preliminaries.}

 We consider the boundary value problem of the form
\begin{equation} \label{1)} 
l\left(y\right):=-y''+\left[2\lambda p\left(x\right)+q\left(x\right)\right]y=\lambda ^{2} \delta \left(x\right)y,\, \, x\in \left[0,\pi \right]/\left\{a_{1} ,a_{2} \right\} 
\end{equation} 
with the boundary conditions
\begin{equation} \label{2)} 
y'\left(0\right)=0,y\left(\pi \right)=0 
\end{equation} 
and the jump conditions
\begin{equation} \label{3)} 
y\left(a_{1} +0\right)=\alpha _{1} y\left(a_{1} -0\right) 
\end{equation} 
\begin{equation} \label{4)} 
y'\left(a_{1} +0\right)=\beta _{1} y'\left(a_{1} -0\right)+i\lambda \gamma _{1} y\left(a_{1} -0\right) 
\end{equation} 
\begin{equation} \label{5)} 
y\left(a_{2} +0\right)=\alpha _{2} y\left(a_{2} -0\right) 
\end{equation} 
\begin{equation} \label{6)} 
y'\left(a_{2} +0\right)=\beta _{2} y'\left(a_{2} -0\right)+i\lambda \gamma _{2} y\left(a_{2} -0\right) 
\end{equation} 
Where $\lambda $ is a spectral parameter, $p(x)\in W_{2}^{1} \left[0,\pi \right]$, $q(x)\in L_{2} \left[0,\pi \right]$ are real valued functions, $a_{1} \in \left[0,\frac{\pi }{2} \right]$, $a_{2} \in \left[\frac{\pi }{2} ,\pi \right]$ , $\alpha _{1} ,\alpha _{2} ,\gamma _{1} ,\gamma _{2} $ are real numbers,  $\left|\alpha _{i} -1\right|^{2} +\gamma _{i} ^{2} \ne 0\, \, \left(\alpha _{i} >0;i=1,2\right)$, $\beta _{i} =\frac{1}{\alpha _{i} } \left(i=1,2\right)$ and\\ $\delta \left(x\right)=\left\{\begin{array}{l} {\alpha ^{2} ,\, \, \, \, x\in \left(0,\frac{\pi }{2} \right)} \\ {\beta ^{2} ,\, \, \, \, x\in \left(\frac{\pi }{2} ,\pi \right)} \end{array}\right. $  where   $0<\alpha <\beta <1$,$\alpha +\beta >1$. 

\noindent The inverse problems consist in recoverint the coefficients of an operator from their spectral characteristics. A lot of study were done the inverse spectral problem for Sturm-Liouville operators and diffusion operators \cite{Acan,Gala,Amirov,Carlson,Ergün-1,F.Yang-1,Gesztes,Hryniv,Huang,Keldysh,Levin,Markus,Sakhnovich,Yang,Yang-1,Yurko,alpay,hald,hochstadt,Gala-1,koyunbakan,levitan,ozkan,wei,10}. The first results an inverse problems theory of Sturm-Liouville operators where given by Ambarzumyan $\left[2\right]$. The half inverse problems for Sturm-Liouville equations; the known potential in half interval is determined by the help of a one spectrum over the interval. First the obtained results the half inverse problem by Hochstadt and Lieberman  \cite{hochstadt}. They proved that spectrum of the problem
\[-y''+q\left(x\right)y=\lambda y,\, \, x\in \left[0,1\right]\] 
\[y'\left(0\right)-hy\left(0\right)=0\] 
\[y'\left(1\right)+Hy\left(1\right)=0\] 
and potential $q\left(x\right)$ on the $\left(\frac{1}{2} ,1\right)$uniquely determine the potential $q\left(x\right)$ on the whole interval $\left[0,1\right]$ almost everywhere. Hald  \cite{hald} proved similar results in the case when there exists a impulsive conditions inside the interval. Many studies have been done by different authours for half invers problems using this methods \cite{koyunbakan,Sakhnovich}. In the work \cite{Sakhnovich} studied the existence of the solution for he half-inverse problem of Sturm-Liouville problems and gave method of reconstructing this solution under same conditions by Sakhnovich  $\left[16\right]$. Recently, same new uniqueness results on the inverse or half inverse spectral analysis of differential operators have been given. Koyunbakan and Panakhov \cite{koyunbakan} proved the half inverse problem for diffusion operator on the finite interval $\left[0,\pi \right]$. Ran Zhang, Xiao-Chuan Xu, Chuan-Fu Yang and Natalia Pavlovna Bondarenko, proved the determination of the impulsive Sturm-Liouville operator from a set
of eigenvalues  \cite{10}  .

\noindent Purpose of this study is to prove half inverse problem by using the Hocstadt- Lieberman and Yang-Zettl methods for the following equations
\begin{equation} \label{7)} 
\tilde{l}\left(y\right):=-y''+\left[2\lambda \tilde{p}\left(x\right)+\tilde{q}\left(x\right)\right]y=\lambda ^{2} \tilde{\delta }\left(x\right)y,\, \, x\in \left[0,\pi \right]/\left\{a_{1} ,a_{2} \right\} 
\end{equation} 
\begin{equation} \label{8)} 
y'\left(0\right)=0,y\left(\pi \right)=0 
\end{equation} 
\begin{equation} \label{9)} 
y\left(a_{1} +0\right)=\tilde{\alpha }_{1} y\left(a_{1} -0\right) 
\end{equation} 
\begin{equation} \label{10)} 
y'\left(a_{1} +0\right)=\tilde{\beta }_{1} y'\left(a_{1} -0\right)+i\lambda \tilde{\gamma }_{1} y\left(a_{1} -0\right) 
\end{equation} 
\begin{equation} \label{11)} 
y\left(a_{2} +0\right)=\tilde{\alpha }_{2} y\left(a_{2} -0\right) 
\end{equation} 
\begin{equation} \label{12)} 
y'\left(a_{2} +0\right)=\tilde{\beta }_{2} y'\left(a_{2} -0\right)+i\lambda \tilde{\gamma }_{2} y\left(a_{2} -0\right). 
\end{equation}

\begin{lemma}\label{lem:1}  Let $p\left(x\right)\in W_{2}^{1} \left(0,\pi \right)$ ,$q\left(x\right)\in L_{2} \left(0,\pi \right)$. \textbf{$M\left(x,t\right)$ },\textbf{$N\left(x,t\right)$} are summable functions on $\left[0,\pi \right]$ such that the representation for each $x\in \left[0,\pi \right]/\left\{a_{1} ,a_{2} \right\}$. $\; \varphi \left(x,\lambda \right)$ solution of the equations $\left(1.1\right)$ , providing boundary conditions $\left(1.2\right)$  and discontinuity conditions $\left(1.3\right)-\left(1.6\right)$
\[\varphi \left(x,\lambda \right)=\varphi _{0} \left(x,\lambda \right)+\int _{0}^{x}M\left(x,t\right) \cos \lambda tdt+\int _{0}^{x}N\left(x,t\right) \sin \lambda tdt\] 
is satisfied, 

\noindent for $0<x<\frac{\pi }{2} $,  
\begin{equation} \label{13)} 
\begin{array}{l} {\varphi _{0} \left(x,\lambda \right)=} \\ {\left(\beta _{1} ^{+} +\frac{\gamma _{1} }{2\alpha } \right)\cos \left[\lambda \xi ^{+} \left(x\right)-\frac{1}{\alpha } \int _{a_{1} }^{x}p\left(t\right)dt \right]+\left(\beta _{1} ^{-} -\frac{\gamma _{1} }{2\alpha } \right)\cos \left[\lambda \xi ^{-} \left(x\right)+\frac{1}{\alpha } \int _{a_{1} }^{x}p\left(t\right)dt \right]} \end{array} 
\end{equation} 
for $\frac{\pi }{2} <x\le \pi $,
\begin{equation} \label{14)} 
\begin{array}{l} {\varphi _{0} \left(x,\lambda \right)=\left(\beta _{2} ^{+} +\frac{\gamma _{2} }{2\beta } \right)\cos \left[\lambda k^{+} \left(\pi \right)-\frac{1}{\beta } \int _{a_{2} }^{\pi }p\left(t\right)dt \right]} \\ {+\left(\beta _{2} ^{-} +\frac{\gamma _{2} }{2\beta } \right)\cos \left[\lambda k^{-} \left(\pi \right)-\frac{1}{\beta } \int _{a_{2} }^{\pi }p\left(t\right)dt \right]} \\ {+\left(\beta _{2} ^{-} -\frac{\gamma _{2} }{2\beta } \right)\cos \left[\lambda s^{+} \left(\pi \right)+\frac{1}{\beta } \int _{a_{2} }^{\pi }p\left(t\right)dt \right]} \\ {+\left(\beta _{2} ^{+} -\frac{\gamma _{2} }{2\beta } \right)\cos \left[\lambda s^{-} \left(\pi \right)+\frac{1}{\beta } \int _{a_{2} }^{\pi }p\left(t\right)dt \right]} \end{array} 
\end{equation} 
where $\xi ^{\pm } \left(x\right)=\pm \alpha x\mp \alpha a_{1} +a_{1} $ , $k^{\pm } \left(x\right)=\xi ^{+} \left(a_{2} \right)\pm \beta x\mp \beta a_{2} $,\\ $s^{\pm } \left(x\right)=\xi ^{-} \left(a_{2} \right)\pm \beta x\mp \beta a_{2} $,$\beta _{1} ^{\mp } =\frac{1}{2} \left(\alpha _{1} \mp \frac{\beta _{1} }{\alpha } \right)$ , $\beta _{2} ^{\mp } =\frac{1}{2} \left(\alpha _{2} \mp \frac{\alpha \beta _{2} }{\beta } \right)$ .

\noindent Thus, following the relations hold;

\noindent If $p\left(x\right)\in W_{2}^{2} \left(0,\pi \right),\, q\left(x\right)\in W_{2}^{1} \left(0,\pi \right)$
\[\left\{\begin{array}{l} {\frac{\partial ^{2} M\left(x,t\right)}{\partial x^{2} } -\rho \left(x\right)\frac{\partial ^{2} M\left(x,t\right)}{\partial t^{2} } =2p\left(x\right)\frac{\partial N\left(x,t\right)}{\partial t} +q\left(x\right)M\left(x,t\right)} \\ {\frac{\partial ^{2} N\left(x,t\right)}{\partial x^{2} } -\rho \left(x\right)\frac{\partial ^{2} N\left(x,t\right)}{\partial t^{2} } =-2p\left(x\right)\frac{\partial M\left(x,t\right)}{\partial t} +q\left(x\right)N\left(x,t\right)} \end{array}\right. \, \] 
\[M\left(x,\varsigma ^{+} \left(x\right)\right)\cos \frac{\beta \left(x\right)}{\alpha } +N\left(x,\varsigma ^{+} \left(x\right)\right)\sin \frac{\beta \left(x\right)}{\alpha } =\left(\beta _{1} ^{+} +\frac{\gamma _{1} }{2\alpha } \right)\int _{0}^{x}\left(q\left(t\right)+\frac{p^{2} \left(t\right)}{\alpha ^{2} } \right) dt\, \] 
\[M\left(x,\varsigma ^{+} \left(x\right)\right)\sin \frac{\beta \left(x\right)}{\alpha } -N\left(x,\varsigma ^{+} \left(x\right)\right)\cos \frac{\beta \left(x\right)}{\alpha } =\left(\beta _{1} ^{+} +\frac{\gamma _{1} }{2\alpha } \right)\left(p\left(x\right)-p\left(0\right)\right)\, \] 
\[\begin{array}{l} {M\left(x,k^{+} \left(x\right)+0\right)-M\left(x,k^{+} \left(x\right)-0\right)=} \\ {-\left(\beta _{2} ^{+} +\frac{\gamma _{2} }{2\beta } \right)\left(p\left(x\right)-p\left(0\right)\right)\, \sin \frac{\omega \left(x\right)}{\beta } -\left(\beta _{2} ^{+} +\frac{\gamma _{2} }{2\beta } \right)\int _{0}^{x}\left(q\left(t\right)+\frac{p^{2} \left(t\right)}{\beta ^{2} } \right) dt\, \cos \frac{\omega \left(x\right)}{\beta } } \end{array}\] 
\[\begin{array}{l} {N\left(x,k^{+} \left(x\right)+0\right)-N\left(x,k^{+} \left(x\right)-0\right)=} \\ {\left(\beta _{2} ^{+} +\frac{\gamma _{2} }{2\beta } \right)\left(p\left(x\right)-p\left(0\right)\right)\, \cos \frac{\omega \left(x\right)}{\beta } -\left(\beta _{2} ^{+} +\frac{\gamma _{2} }{2\beta } \right)\int _{0}^{x}\left(q\left(t\right)+\frac{p^{2} \left(t\right)}{\beta ^{2} } \right) dt\, \sin \frac{\omega \left(x\right)}{\beta } } \end{array}\] 

\[\left. \frac{\partial M\left(x,t\right)}{\partial t} \right|_{t=0} =N\left(x,0\right)=0\] 
where $\beta \left(x\right)=\int _{0}^{x}p\left(t\right)dt $,$\omega \left(x\right)=\int _{a_{2} }^{x}p\left(t\right)dt +\int _{0}^{a_{1} }p\left(t\right)dt $.

\noindent The proof is done as in \cite{Ergün-1}.
\end{lemma}

\textbf{Definition.} The function $\Delta \left(\lambda \right)$ is called the characteristic function of the eigenvalues $\left\{\lambda _{n} \right\}$of the problem $\left(1.1\right)-\left(1.6\right)$. $\tilde{\Delta }\left(\lambda \right)$ is called the characteristic function of the eigenvalues $\left\{\tilde{\lambda }_{n} \right\}$of the problem $\left(1.7\right)-\left(1.12\right)$. 

\noindent Let $\lambda =s^{2} ,s=\sigma +i\tau \, ,\, \sigma ,\tau \in {\rm R}$. The solution $\varphi \left(x,\lambda \right)$ of $\left(1.1\right)-\left(1.6\right)$ have the following asymptotic formulas hold on for $\left|\lambda \right|\to \infty $,

\noindent for $0<x<\frac{\pi }{2} $, 
\[\varphi \left(x,\lambda \right)=\frac{1}{2} \left(\frac{\alpha _{1} }{2} \mp \frac{\beta _{1} }{2\alpha } +\frac{\gamma _{1} }{2\alpha } \right)\exp \left(-i\left(\lambda \xi ^{+} \left(x\right)-\frac{v\left(x\right)}{\alpha } \right)\right)\left(1+O\left(\frac{1}{\lambda } \right)\right)\] 
for $\frac{\pi }{2} <x\le \pi $ ,
\[\varphi \left(x,\lambda \right)=\frac{1}{2} \left(\frac{\alpha _{2} }{2} +\frac{\alpha \beta _{2} }{2\beta } +\frac{\gamma _{2} }{2\beta } \right)\exp \left(-i\left(\lambda k^{+} \left(x\right)-\frac{t\left(x\right)}{\beta } \right)\right)\left(1+O\left(\frac{1}{\lambda } \right)\right).\] 
where $v\left(x\right)=\int _{a_{1} }^{x}p\left(t\right)dt $, $t\left(x\right)=\int _{a_{2} }^{x}p\left(t\right)dt $.

\noindent 

\noindent 

\noindent In this study, if $q\left(x\right)$ and $p\left(x\right)$ to be known almost everywhere $\left(\frac{\pi }{2} ,\pi \right)$,  sufficient to determine uniquely $p\left(x\right)$ and $q\left(x\right)$ whole interval $\left(0,\pi \right)$ .

\section{main result}

\noindent If $\varphi_{0} \left(x,\lambda \right)$ a nontrivial solution of equation $\left(1.1\right)$ with conditions $\left(1.2\right)$-$\left(1.6\right)$, then $\lambda _{0} $ is called eigenvalue. Additionally, $\varphi_{0} \left(x,\lambda \right)$ is called the eigenfunction of the problem corresponding to the eigenvalue $\lambda _{0} $. $\left\{\lambda _{n} \right\}$ are eigenvalues of the problem.
\begin{lemma}\label{lem:1} If $\lambda _{n} =\tilde{\lambda }_{n} $, $\frac{\alpha }{\tilde{\alpha }} =\frac{\beta }{\tilde{\beta }} $ then $\alpha =\tilde{\alpha }$ and $\beta =\tilde{\beta }$ for all $n\in {\rm N}$.
\end{lemma}
\begin{proof}

\noindent Since $\lambda _{n} =\tilde{\lambda }_{n} $ and $\Delta \left(\lambda \right),\, \tilde{\Delta }\left(\lambda \right)$are entire functions in $\lambda $ of order one by Hadamard factorization theorem for $\lambda \in {\rm C}$
\[\Delta \left(\lambda \right)\equiv C\, \tilde{\Delta }\left(\lambda \right)\] 
On the other hand, $\left(1.1\right)$ can be written as
\[\Delta _{0} \left(\lambda \right)-C\, \tilde{\Delta }_{0} \left(\lambda \right)=C\left[\tilde{\Delta }\left(\lambda \right)-\, \tilde{\Delta }_{0} \left(\lambda \right)\right]-\left[\Delta \left(\lambda \right)-\, \Delta _{0} \left(\lambda \right)\right]\] 
Hence
\begin{equation} \label{15)} 
\begin{array}{l} {C\left[\tilde{\Delta }\left(\lambda \right)-\, \tilde{\Delta }_{0} \left(\lambda \right)\right]-\left[\Delta \left(\lambda \right)-\, \Delta _{0} \left(\lambda \right)\right]=} \\ {\left(\beta _{2} ^{+} +\frac{\gamma _{2} }{2\beta } \right)\cos \left[\lambda k^{+} \left(\pi \right)-\frac{w\left(\pi \right)}{\beta } \right]+\left(\beta _{2} ^{-} +\frac{\gamma _{2} }{2\beta } \right)\cos \left[\lambda k^{-} \left(\pi \right)-\frac{w\left(\pi \right)}{\beta } \right]} \\ {+\left(\beta _{2} ^{-} -\frac{\gamma _{2} }{2\beta } \right)\cos \left[\lambda s^{+} \left(\pi \right)+\frac{w\left(\pi \right)}{\beta } \right]+\left(\beta _{2} ^{+} -\frac{\gamma _{2} }{2\beta } \right)\cos \left[\lambda s^{-} \left(\pi \right)+\frac{w\left(\pi \right)}{\beta } \right]} \\ {-C\left(\tilde{\beta }_{2} ^{+} +\frac{\tilde{\gamma }_{2} }{2\tilde{\beta }} \right)\cos \left[\lambda k^{+} \left(\pi \right)-\frac{\tilde{w}\left(\pi \right)}{\tilde{\beta }} \right]-C\left(\tilde{\beta }_{2} ^{-} +\frac{\tilde{\gamma }_{2} }{2\tilde{\beta }} \right)\cos \left[\lambda k^{-} \left(\pi \right)-\frac{\tilde{w}\left(\pi \right)}{\tilde{\beta }} \right]} \\ {-C\left(\tilde{\beta }_{2} ^{-} -\frac{\tilde{\gamma }_{2} }{2\tilde{\beta }} \right)\cos \left[\lambda s^{+} \left(\pi \right)+\frac{\tilde{w}\left(\pi \right)}{\tilde{\beta }} \right]-C\left(\tilde{\beta }_{2} ^{+} -\frac{\tilde{\gamma }_{2} }{2\tilde{\beta }} \right)\cos \left[\lambda s^{-} \left(\pi \right)+\frac{\tilde{w}\left(\pi \right)}{\beta } \right]} \end{array} 
\end{equation} 
If we multiply both sides of $\left(2.1\right)$ by $\cos \left[\lambda k^{+} \left(\pi \right)-\frac{w\left(\pi \right)}{\beta } \right]$ and integrate with respect to $\lambda $ in $\left(\varepsilon ,T\right)$, ($\varepsilon $ is sufficiently small positive number) for any positive real number $T$, then we get

\noindent 
\[\begin{array}{l} {\int _{\varepsilon }^{T}\left(C\left[\tilde{\Delta }\left(\lambda \right)-\, \tilde{\Delta }_{0} \left(\lambda \right)\right]-\left[\Delta \left(\lambda \right)-\, \Delta _{0} \left(\lambda \right)\right]\right)\cos \left[\lambda k^{+} \left(\pi \right)-\frac{w\left(\pi \right)}{\beta } \right] d\lambda =} \\ {+\int _{\varepsilon }^{T}\left\{\left(\beta _{2} ^{+} +\frac{\gamma _{2} }{2\beta } \right)\cos \left[\lambda k^{+} \left(\pi \right)-\frac{w\left(\pi \right)}{\beta } \right]+\left(\beta _{2} ^{-} +\frac{\gamma _{2} }{2\beta } \right)\cos \right. \left[\lambda k^{-} \left(\pi \right)-\frac{w\left(\pi \right)}{\beta } \right] } \\ {+\left(\beta _{2} ^{-} -\frac{\gamma _{2} }{2\beta } \right)\cos \left[\lambda s^{+} \left(\pi \right)+\frac{w\left(\pi \right)}{\beta } \right]+\left(\beta _{2} ^{+} -\frac{\gamma _{2} }{2\beta } \right)\cos \left[\lambda s^{-} \left(\pi \right)+\frac{w\left(\pi \right)}{\beta } \right]} \\ {-C\left(\tilde{\beta }_{2} ^{+} +\frac{\tilde{\gamma }_{2} }{2\tilde{\beta }} \right)\cos \left[\lambda k^{+} \left(\pi \right)-\frac{\tilde{w}\left(\pi \right)}{\tilde{\beta }} \right]-C\left(\tilde{\beta }_{2} ^{-} +\frac{\tilde{\gamma }_{2} }{2\tilde{\beta }} \right)\cos \left[\lambda k^{-} \left(\pi \right)-\frac{\tilde{w}\left(\pi \right)}{\tilde{\beta }} \right]} \\ {\left. -C\left(\tilde{\beta }_{2} ^{-} -\frac{\tilde{\gamma }_{2} }{2\tilde{\beta }} \right)\cos \left[\lambda s^{+} \left(\pi \right)+\frac{\tilde{w}\left(\pi \right)}{\tilde{\beta }} \right]-C\left(\tilde{\beta }_{2} ^{+} -\frac{\tilde{\gamma }_{2} }{2\tilde{\beta }} \right)\cos \left[\lambda s^{-} \left(\pi \right)+\frac{\tilde{w}\left(\pi \right)}{\beta } \right]\right\}d\lambda } \end{array}\] 
And so
\[\begin{array}{l} {\int _{\varepsilon }^{T}\left(C\left[\tilde{\Delta }\left(\lambda \right)-\, \tilde{\Delta }_{0} \left(\lambda \right)\right]-\left[\Delta \left(\lambda \right)-\, \Delta _{0} \left(\lambda \right)\right]\right)\cos \left[\lambda k^{+} \left(\pi \right)-\frac{w\left(\pi \right)}{\beta } \right] d\lambda =} \\ {\int _{\varepsilon }^{T}\left(\beta _{2} ^{+} +\frac{\gamma _{2} }{2\beta } \right)\cos ^{2} \left[\lambda k^{+} \left(\pi \right)-\frac{w\left(\pi \right)}{\beta } \right] d\lambda } \\ {-C\int _{\varepsilon }^{T}\left(\tilde{\beta }_{2} ^{+} +\frac{\tilde{\gamma }_{2} }{2\tilde{\beta }} \right)\cos \left[\lambda k^{+} \left(\pi \right)-\frac{w\left(\pi \right)}{\beta } \right]\cos \left[\lambda k^{+} \left(\pi \right)-\frac{\tilde{w}\left(\pi \right)}{\tilde{\beta }} \right] d\lambda } \end{array}\] 
\[\begin{array}{l} {=\int _{\varepsilon }^{T}\frac{1}{2} \left(\beta _{2} ^{+} +\frac{\gamma _{2} }{2\beta } \right)+\frac{1}{2} \left(\beta _{2} ^{+} +\frac{\gamma _{2} }{2\beta } \right)\cos \left[2\lambda k^{+} \left(\pi \right)-\frac{2w\left(\pi \right)}{\beta } \right] d\lambda } \\ {-C\int _{\varepsilon }^{T}\frac{1}{2} \left(\tilde{\beta }_{2} ^{+} +\frac{\tilde{\gamma }_{2} }{2\tilde{\beta }} \right)\left(\cos \left[2\lambda k^{+} \left(\pi \right)-\frac{\tilde{w}\left(\pi \right)+w\left(\pi \right)}{\beta } \right]+\cos \left[\frac{w\left(\pi \right)-\tilde{w}\left(\pi \right)}{\tilde{\beta }} \right]\right) d\lambda } \end{array}\] 
$\Delta \left(\lambda \right)-\Delta _{0} \left(\lambda \right)=O\left(\frac{1}{\left|\lambda \right|} e^{\left|Im\lambda \right|k^{+} \left(\pi \right)} \right)$, $\tilde{\Delta }\left(\lambda \right)-\tilde{\Delta }_{0} \left(\lambda \right)=O\left(\frac{1}{\left|\lambda \right|} e^{\left|Im\lambda \right|k^{+} \left(\pi \right)} \right)$ for all $\lambda $ in $\left(\varepsilon ,T\right)$.
\[\frac{C}{2} \left(\tilde{\beta }_{2} ^{+} +\frac{\tilde{\gamma }_{2} }{2\tilde{\beta }} \right)-\frac{1}{2} \left(\beta _{2} ^{+} +\frac{\gamma _{2} }{2\beta } \right)=O\left(\frac{1}{T} \right)\] 
By letting $T$ tend to infinity we see that
\[C=\frac{\tilde{\beta }_{2} ^{+} +\frac{\tilde{\gamma }_{2} }{2\tilde{\beta }} }{\beta _{2} ^{+} +\frac{\gamma _{2} }{2\beta } } \] 
Similarly, if we multiply both side of $\left(2.1\right)$ $\cos \left[\lambda k^{-} \left(\pi \right)-\frac{w\left(\pi \right)}{\beta } \right]$ and integrate again with respect to $\lambda $ in $\left(\varepsilon ,T\right)$ and by letting $T$ tend to infinity, then we get
\[C=\frac{\tilde{\beta }_{2} ^{-} +\frac{\tilde{\gamma }_{2} }{2\tilde{\beta }} }{\beta _{2} ^{-} +\frac{\gamma _{2} }{2\beta } } \] 
But since $\alpha ,\beta $ and $\tilde{\alpha },\tilde{\beta }$ are positive, since $w^{+} \left(\pi \right)-\tilde{w}^{+} \left(\pi \right)=w^{-} \left(\pi \right)-\tilde{w}^{-} \left(\pi \right)$ we conclude that $C=1$. Hence $\frac{\tilde{\beta }_{2} ^{+} }{\beta _{2} ^{+} } =\frac{\tilde{\beta }_{2} ^{-} }{\beta _{2} ^{-} } $ is obtained. We have therefore proved since $\alpha =\tilde{\alpha }$ that $\beta =\tilde{\beta }$.

\noindent The proof is completed.
\end{proof}

\begin{lemma}\label{lem:1}If $\lambda _{n} =\tilde{\lambda }_{n} $ then $\alpha _{i} =\tilde{\alpha }_{i} $ and $\gamma _{i} =\tilde{\gamma }_{i} $ $\left(i=1,2\right)$for all $n\in {\rm N}$.

\noindent The proof is done as in \cite{Ergün-1}.
\end{lemma}
\noindent

\begin{theorem}\label{1} Let $\left\{\lambda _{n} \right\}$ a eigenvalues of both problem $\left(1.1\right)-\left(1.6\right)$ and $\left(1.7\right)-\left(1.12\right)$. If $p\left(x\right)=\tilde{p}\left(x\right)$ and $q\left(x\right)=\tilde{q}\left(x\right)$ on $\left[\frac{\pi }{2} ,\pi \right]$ , then $p\left(x\right)=\tilde{p}\left(x\right)$ and $q\left(x\right)=\tilde{q}\left(x\right)$ almost everywhere on $\left[0,\pi \right]$.
\end{theorem}

\begin{proof}[Proof of Theorem {1}] Let function $\varphi \left(x,\lambda \right)$ the solution of equation $\left(1.1\right)$ under the conditions $\left(1.2\right)-\left(1.6\right)$ and the function $\tilde{\varphi }\left(x,\lambda \right)$ the solution of equation $\left(1.7\right)$ under the conditions $\left(1.8\right)-\left(1.12\right)$in $\left[0,\frac{\pi }{2} \right]$. The integral forms of the  functions $\varphi \left(x,\lambda \right)$ and $\tilde{\varphi }\left(x,\lambda \right)$ can be obtained as follows
	
	\noindent 
	\begin{equation} \label{16)} 
	\begin{array}{l} {\varphi \left(x,\lambda \right)=\left(\beta _{1} ^{+} +\frac{\gamma _{1} }{2\alpha } \right)\cos \left[\lambda \xi ^{+} \left(x\right)-\frac{1}{\alpha } \int _{a_{1} }^{x}p\left(t\right)dt \right]} \\ {+\left(\beta _{1} ^{-} -\frac{\gamma _{1} }{2\alpha } \right)\cos \left[\lambda \xi ^{-} \left(x\right)+\frac{1}{\alpha } \int _{a_{1} }^{x}p\left(t\right)dt \right]+\int _{0}^{x}M\left(x,t\right)\cos \lambda tdt +\int _{0}^{x}N\left(x,t\right)\sin \lambda tdt } \end{array} 
	\end{equation} 
	and
	\begin{equation} \label{17)} 
	\begin{array}{l} {\tilde{\varphi }\left(x,\lambda \right)=\left(\tilde{\beta }_{1} ^{+} +\frac{\tilde{\gamma }_{1} }{2\alpha } \right)\cos \left[\lambda \xi ^{+} \left(x\right)-\frac{1}{\alpha } \int _{a_{1} }^{x}\tilde{p}\left(t\right)dt \right]} \\ {+\left(\tilde{\beta }_{1} ^{-} -\frac{\tilde{\gamma }_{1} }{2\alpha } \right)\cos \left[\lambda \xi ^{-} \left(x\right)+\frac{1}{\alpha } \int _{a_{1} }^{x}\tilde{p}\left(t\right)dt \right]+\int _{0}^{x}\tilde{M}\left(x,t\right)\cos \lambda tdt +\int _{0}^{x}\tilde{N}\left(x,t\right)\sin \lambda tdt } \end{array} 
	\end{equation} 
	If we multiply equations $\left(2.2\right)$ and $\left(2.3\right)$\\
$\begin{array}{l} {\varphi \left(x,\lambda \right)\cdot \tilde{\varphi }\left(x,\lambda \right)=\frac{S^{+} \tilde{S}^{+} }{2} \left[\cos \left(2\lambda \xi ^{+} \left(x\right)-K\left(x\right)\right)+\cos L\left(x\right)\right]} \\ {+\frac{S^{+} \tilde{S}^{-} }{2} \left[\cos \left(2\lambda a_{1} t-L\left(x\right)\right)+\cos \left(2\lambda \alpha \left(x-a_{1} \right)-K\left(x\right)\right)\right]} \\ {+\frac{S^{-} \tilde{S}^{+} }{2} \left[\cos \left(2\lambda a_{1} +L\left(x\right)\right)+\cos \left(2\lambda \alpha \left(x-a_{1} \right)+K\left(x\right)\right)\right]} \\ {+\frac{S^{-} \tilde{S}^{-} }{2} \left[\cos \left(2\lambda \xi ^{-} \left(x\right)+L\left(x\right)\right)+\cos K\left(x\right)\right]} \\ {+S^{+} \int _{0}^{x}\tilde{M}\left(x,t\right)\cos \left[\lambda \xi ^{+} \left(x\right)-\frac{t\left(x\right)}{\alpha } \right] \cos \lambda tdt} \\ {+S^{+} \int _{0}^{x}\tilde{N}\left(x,t\right)\cos \left[\lambda \xi ^{+} \left(x\right)-\frac{t\left(x\right)}{\alpha } \right] \sin \lambda tdt} \\ {+S^{-} \int _{0}^{x}\tilde{M}\left(x,t\right)\cos \left[\lambda \xi ^{-} \left(x\right)+\frac{t\left(x\right)}{\alpha } \right] \cos \lambda tdt} \\ {+S^{-} \int _{0}^{x}\tilde{N}\left(x,t\right)\cos \left[\lambda \xi ^{-} \left(x\right)+\frac{t\left(x\right)}{\alpha } \right] \sin \lambda tdt} \\ {+\tilde{S}^{+} \int _{0}^{x}M\left(x,t\right)\cos \left[\lambda \xi ^{+} \left(x\right)-\frac{\tilde{t}\left(x\right)}{\alpha } \right] \cos \lambda tdt} \\ {+\tilde{S}^{+} \int _{0}^{x}N\left(x,t\right)\cos \left[\lambda \xi ^{+} \left(x\right)-\frac{\tilde{t}\left(x\right)}{\alpha } \right] \sin \lambda tdt} \\ {+\tilde{S}^{-} \int _{0}^{x}M\left(x,t\right)\cos \left[\lambda \xi ^{-} \left(x\right)+\frac{\tilde{t}\left(x\right)}{\alpha } \right] \cos \lambda tdt} \\ {+\tilde{S}^{-} \int _{0}^{x}N\left(x,t\right)\cos \left[\lambda \xi ^{-} \left(x\right)+\frac{\tilde{t}\left(x\right)}{\alpha } \right] \sin \lambda tdt} \\ {+\left(\int _{0}^{x}M\left(x,t\right)\cos \lambda tdt \right)\left(\int _{0}^{x}\tilde{M}\left(x,t\right)\cos \lambda tdt \right)} \\ {+\left(\int _{0}^{x}N\left(x,t\right)\sin \lambda tdt \right)\left(\int _{0}^{x}\tilde{N}\left(x,t\right)\sin \lambda tdt \right)} \\ {+\left(\int _{0}^{x}M\left(x,t\right)\cos \lambda tdt \right)\left(\int _{0}^{x}\tilde{N}\left(x,t\right)\sin \lambda tdt \right)} \\ {+\left(\int _{0}^{x}\tilde{M}\left(x,t\right)\cos \lambda tdt \right)\left(\int _{0}^{x}N\left(x,t\right)\sin \lambda tdt \right)} \end{array}$

	\noindent 
	\begin{equation} \label{18)} 
	\begin{array}{l} {\varphi \left(x,\lambda \right)\cdot \tilde{\varphi }\left(x,\lambda \right)=\frac{S^{+} \tilde{S}^{+} }{2} \left[\cos \left(2\lambda \xi ^{+} \left(x\right)-K\left(x\right)\right)+\cos L\left(x\right)\right]} \\ {+\frac{S^{+} \tilde{S}^{-} }{2} \left[\cos \left(2\lambda a_{1} t-L\left(x\right)\right)+\cos \left(2\lambda \alpha \left(x-a_{1} \right)-K\left(x\right)\right)\right]} \\ {+\frac{S^{-} \tilde{S}^{+} }{2} \left[\cos \left(2\lambda a_{1} +L\left(x\right)\right)+\cos \left(2\lambda \alpha \left(x-a_{1} \right)+K\left(x\right)\right)\right]} \\ {+\frac{S^{-} \tilde{S}^{-} }{2} \left[\cos \left(2\lambda \xi ^{-} \left(x\right)+L\left(x\right)\right)+\cos K\left(x\right)\right]} \\ {+\frac{1}{2} \left\{\int _{0}^{x}U_{c} \left(x,t\right)\cos \left(2\lambda t-K\left(t\right)\right)dt -\int _{0}^{x}U_{s} \left(x,t\right)\sin \left(2\lambda t-K\left(t\right)\right)dt \right\}} \end{array} 
	\end{equation} 
	is obtained, being $S^{\pm } =\left(\beta _{1} ^{\pm } \mp \frac{\gamma _{1} }{2\alpha } \right)$, $\tilde{S}^{\pm } =\left(\tilde{\beta }_{1} ^{\pm } \mp \frac{\tilde{\gamma }_{1} }{2\alpha } \right)$, $K\left(x\right)=\frac{t\left(x\right)+\tilde{t}\left(x\right)}{2} $, $L\left(x\right)=\frac{t\left(x\right)-\tilde{t}\left(x\right)}{2} $,\\
$\begin{array}{l} {U_{c} \left(x,t\right)=S^{+} \tilde{M}\left(x,\xi ^{+} \left(x\right)-2t\right)\cos \left(K\left(t\right)-\frac{t\left(x\right)}{\alpha } \right)} \\ {+S^{-} \tilde{M}\left(x,\xi ^{-} \left(x\right)-2t\right)\cos \left(K\left(t\right)-\frac{t\left(x\right)}{\alpha } \right)} \\ {+\tilde{S}^{+} M\left(x,\xi ^{+} \left(x\right)-2t\right)\cos \left(K\left(t\right)-\frac{\tilde{t}\left(x\right)}{\alpha } \right)} \\ {+\tilde{S}^{-} M\left(x,\xi ^{-} \left(x\right)-2t\right)\sin \left(K\left(t\right)-\frac{\tilde{t}\left(x\right)}{\alpha } \right)} \\ {-S^{-} \tilde{N}\left(x,\xi ^{+} \left(x\right)-2t\right)\sin \left(K\left(t\right)-\frac{t\left(x\right)}{\alpha } \right)} \\ {-S^{-} \tilde{N}\left(x,\xi ^{-} \left(x\right)-2t\right)\sin \left(K\left(t\right)-\frac{t\left(x\right)}{\alpha } \right)} \\ {-\tilde{S}^{+} N\left(x,\xi ^{+} \left(x\right)-2t\right)\sin \left(K\left(t\right)-\frac{\tilde{t}\left(x\right)}{\alpha } \right)} \\ {-\tilde{S}^{-} N\left(x,\xi ^{-} \left(x\right)-2t\right)\sin \left(K\left(t\right)-\frac{\tilde{t}\left(x\right)}{\alpha } \right)} \\ {+K_{1} \left(x,t\right)\cos K\left(t\right)+K_{2} \left(x,t\right)\cos K\left(t\right)} \\ {+M_{1} \left(x,t\right)\sin K\left(t\right)+M_{2} \left(x,t\right)\sin K\left(t\right)} \end{array}$ \\
	
$\begin{array}{l} {U_{s} \left(x,t\right)=S^{+} \tilde{M}\left(x,\xi ^{+} \left(x\right)-2t\right)\sin \left(K\left(t\right)-\frac{t\left(x\right)}{\alpha } \right)} \\ {+S^{-} \tilde{M}\left(x,\xi ^{-} \left(x\right)-2t\right)\sin \left(K\left(t\right)-\frac{t\left(x\right)}{\alpha } \right)} \\ {+\tilde{S}^{+} M\left(x,\xi ^{+} \left(x\right)-2t\right)\sin \left(K\left(t\right)-\frac{\tilde{t}\left(x\right)}{\alpha } \right)} \\ {+\tilde{S}^{-} M\left(x,\xi ^{-} \left(x\right)-2t\right)\sin \left(K\left(t\right)-\frac{\tilde{t}\left(x\right)}{\alpha } \right)} \\ {+S^{+} \tilde{N}\left(x,\xi ^{+} \left(x\right)-2t\right)\cos \left(K\left(t\right)-\frac{t\left(x\right)}{\alpha } \right)} \\ {+S^{-} \tilde{N}\left(x,\xi ^{-} \left(x\right)-2t\right)\cos \left(K\left(t\right)-\frac{t\left(x\right)}{\alpha } \right)} \\ {+\tilde{S}^{+} N\left(x,\xi ^{+} \left(x\right)-2t\right)\cos \left(K\left(t\right)-\frac{\tilde{t}\left(x\right)}{\alpha } \right)} \\ {+\tilde{S}^{-} N\left(x,\xi ^{-} \left(x\right)-2t\right)\cos \left(K\left(t\right)-\frac{\tilde{t}\left(x\right)}{\alpha } \right)} \\ {+K_{1} \left(x,t\right)\sin K\left(t\right)+K_{2} \left(x,t\right)\sin K\left(t\right)} \\ {-M_{1} \left(x,t\right)\cos K\left(t\right)-M_{2} \left(x,t\right)\cos K\left(t\right)} \end{array}$
	\[K_{1} \left(x,t\right)=\int _{-x}^{x-2t}M\left(x,s\right)\tilde{M}\left(x,s+2t\right) ds+\int _{2t-x}^{x}M\left(x,s\right)\tilde{M}\left(x,s+2t\right) ds\] 
	\[K_{2} \left(x,t\right)=\int _{-x}^{x-2t}N\left(x,s\right)\tilde{N}\left(x,s+2t\right) ds+\int _{2t-x}^{x}n\left(x,s\right)\tilde{N}\left(x,s+2t\right) ds\] 
	\[M_{1} \left(x,t\right)=\int _{-x}^{x-2t}M\left(x,s\right)\tilde{N}\left(x,s+2t\right) ds-\int _{2t-x}^{x}M\left(x,s\right)\tilde{N}\left(x,s+2t\right) ds\] 
	\[M_{2} \left(x,t\right)=-\int _{-x}^{x-2t}N\left(x,s\right)\tilde{M}\left(x,s+2t\right) ds+\int _{2t-x}^{x}N\left(x,s\right)\tilde{M}\left(x,s+2t\right) ds\] 
	Let $\varphi \left(x,\lambda \right)$ and $\tilde{\varphi }\left(x,\lambda \right)$ are substituted into $\left(1.1\right)$ and $\left(1.7\right)$,  
\begin{equation} \label{19)} 
	-\varphi ''\left(x,\lambda \right)+\left(2\lambda p\left(x\right)+q\left(x\right)\right)\varphi \left(x,\lambda \right)=\lambda ^{2} \rho \left(x\right)\varphi \left(x,\lambda \right) 
\end{equation} 
	\begin{equation} \label{20)} 
	-\tilde{\varphi }''\left(x,\lambda \right)+\left(2\lambda p\left(x\right)+q\left(x\right)\right)\tilde{\varphi }\left(x,\lambda \right)=\lambda ^{2} \rho \left(x\right)\tilde{\varphi }\left(x,\lambda \right) 
	\end{equation} 
	The following equations is obtained  $\left(2.5\right)$ and $\left(2.6\right)$ \\
$\begin{array}{l} {\int _{0}^{\frac{\pi }{2} }\varphi \left(x,\lambda \right)\tilde{\varphi }\left(x,\lambda \right)\left[2\lambda \left(p\left(x\right)-\tilde{p}\left(x\right)\right)+\left(q\left(x\right)-\tilde{q}\left(x\right)\right)\right] dx} \\ {=\left[\tilde{\varphi }'\left(x,\lambda \right)\varphi \left(x,\lambda \right)-\varphi '\left(x,\lambda \right)\tilde{\varphi }\left(x,\lambda \right)\right]_{0}^{\frac{\pi }{2} } +\left. \right|_{\frac{\pi }{2} }^{\pi } } \end{array}$ 
	\begin{equation} \label{21)} 
	\begin{array}{l} {\int _{0}^{\frac{\pi }{2} }\varphi \left(x,\lambda \right)\tilde{\varphi }\left(x,\lambda \right)\left[2\lambda \left(p\left(x\right)-\tilde{p}\left(x\right)\right)+\left(q\left(x\right)-\tilde{q}\left(x\right)\right)\right] dx} \\ {+\tilde{\varphi }'\left(\pi ,\lambda \right)\varphi \left(\pi ,\lambda \right)-\varphi '\left(\pi ,\lambda \right)\tilde{\varphi }\left(\pi ,\lambda \right)=0} \end{array} 
	\end{equation} 
	Let $Q\left(x\right)=q\left(x\right)-\tilde{q}\left(x\right)$ and $P\left(x\right)=p\left(x\right)-\tilde{p}\left(x\right)$ 
	\[U\left(\lambda \right)=\int _{0}^{\frac{\pi }{2} }\left[2\lambda P\left(x\right)+Q\left(x\right)\right] \varphi \left(x,\lambda \right)\tilde{\varphi }\left(x,\lambda \right)dx\] 
	It is obvious that the functions $\varphi \left(x,\lambda \right)$ and $\tilde{\varphi }\left(x,\lambda \right)$are the solutions which satisfy boundary value conditions of $\left(1.2\right)$ and $\left(1.8\right)$, recpectively, then if we consider this facts in equation $\left(2.7\right)$, we obtain the following equation
	\begin{equation} \label{22)} 
	U\left(\lambda _{n} \right)=0 
	\end{equation} 
	for each eigenvalue $\lambda _{n} $. Let us marked
	\[U_{1} \left(\lambda \right)=\int _{0}^{\frac{\pi }{2} }P\left(x\right) \varphi \left(x,\lambda \right)\tilde{\varphi }\left(x,\lambda \right)dx, U_{2} \left(\lambda \right)=\int _{0}^{\frac{\pi }{2} }Q\left(x\right) \varphi \left(x,\lambda \right)\tilde{\varphi }\left(x,\lambda \right)dx\] 
	Then equations $\left(2.7\right)$ can be rewritten as 
	\[2\lambda _{n} U_{1} \left(\lambda _{n} \right)+U_{2} \left(\lambda _{n} \right)=0.\] 
	From $\left(2.4\right)$ ve $\left(2.7\right)$ we obtain   
	\begin{equation} \label{23)} 
	\left|U\left(\lambda \right)\right|\le \left(C_{1} +C_{2} \left|\lambda \right|\right)\exp \left(\tau \pi \right) 
	\end{equation} 
	$C_{1} ,C_{2} >0$ are constants.
	
	\noindent For all complex $\lambda $. Because $\lambda _{n} =\tilde{\lambda }_{n} $,  $\Delta \left(\lambda \right)=\varphi \left(\pi ,\lambda \right)=\tilde{\varphi }\left(\pi ,\lambda \right)$. Thus,
	\[U\left(\lambda \right)=\int _{0}^{\frac{\pi }{2} }\left[2\lambda P\left(x\right)+Q\left(x\right)\right] \varphi \left(x,\lambda \right)\tilde{\varphi }\left(x,\lambda \right)dx=\Delta \left(\lambda \right)\left[\varphi \left(\pi ,\lambda \right)-\tilde{\varphi }\left(\pi ,\lambda \right)\right] .\] 
	The function $\phi \left(\lambda \right)=\frac{U\left(\lambda \right)}{\Delta \left(\lambda \right)} $ is an entire function with respect to $\lambda $.
	
	\noindent It follows from $\Delta \left(\lambda \right)\ge \left(\left|\lambda \beta \right|-C\right)\exp \left(\tau \xi ^{+} \left(x\right)\right)$  and $\left(2.9\right)$, $\phi \left(\lambda \right)=O\left(1\right)$ for sufficient large $\left|\lambda \right|$. We obtain  $\phi \left(\lambda \right)=C$, for all $\lambda $ by Liouville's Theorem.\\
$U\left(\lambda \right)=C\Delta \left(\lambda \right)$ \\
$\begin{array}{l} {\int _{0}^{\frac{\pi }{2} }\varphi \left(x,\lambda \right)\tilde{\varphi }\left(x,\lambda \right)\left[2\lambda P\left(x\right)+Q\left(x\right)\right] dx=} \\ {=C\left[\left(\beta _{2} ^{+} +\frac{\gamma _{2} }{2\beta } \right)R_{1} \left(a_{2} \right)\cos \left[\lambda k^{+} \left(\pi \right)-\frac{1}{\beta } \int _{a_{2} }^{\pi }p\left(t\right)dt \right]\right. } \\ {+\left(\beta _{2} ^{-} +\frac{\gamma _{2} }{2\beta } \right)R_{2} \left(a_{2} \right)\cos \left[\lambda k^{-} \left(\pi \right)-\frac{1}{\beta } \int _{a_{2} }^{\pi }p\left(t\right)dt \right]} \\ {+\left(\beta _{2} ^{-} -\frac{\gamma _{2} }{2\beta } \right)R_{1} \left(a_{2} \right)\cos \left[\lambda s^{+} \left(\pi \right)+\frac{1}{\beta } \int _{a_{2} }^{\pi }p\left(t\right)dt \right]} \\ {\left. +\left(\beta _{2} ^{+} -\frac{\gamma _{2} }{2\beta } \right)R_{2} \left(a_{2} \right)\cos \left[\lambda s^{-} \left(\pi \right)+\frac{1}{\beta } \int _{a_{2} }^{\pi }p\left(t\right)dt \right]\right]+O\left(\exp \left(\tau k^{+} \left(\pi \right)\right)\right)} \end{array}$\\
	By the Riemann-Lebesgue lemma, for $\lambda \to \infty $ , $\lambda \in {\rm R}$ we get $C=0$. Then,\\
$\begin{array}{l} {2U_{1} \left(\lambda \right)=S^{+} \tilde{S}^{+} \int _{0}^{\frac{\pi }{2} }P\left(x\right)\cos \left(2\lambda \xi ^{+} \left(x\right)-K\left(x\right)\right)dx } \\ {+S^{+} \tilde{S}^{+} \int _{0}^{\frac{\pi }{2} }P\left(x\right)\cos L\left(x\right) dx} \\ {+S^{+} \tilde{S}^{-} \int _{0}^{\frac{\pi }{2} }P\left(x\right)\cos \left(2\lambda a_{1} t-L\left(x\right)\right)dx } \\ {+S^{+} \tilde{S}^{-} \int _{0}^{\frac{\pi }{2} }P\left(x\right)\cos \left(2\lambda \alpha \left(x-a_{1} \right)-K\left(x\right)\right)dx } \\ {+S^{-} \tilde{S}^{+} \int _{0}^{\frac{\pi }{2} }P\left(x\right)\cos \left(2\lambda a_{1} +L\left(x\right)\right)dx } \\ {+S^{-} \tilde{S}^{+} \int _{0}^{\frac{\pi }{2} }P\left(x\right)\cos \cos \left(2\lambda \alpha \left(x-a_{1} \right)+K\left(x\right)\right)dx } \\ {+S^{-} \tilde{S}^{-} \int _{0}^{\frac{\pi }{2} }P\left(x\right)\cos \left(2\lambda \xi ^{-} \left(x\right)+L\left(x\right)\right)dx } \\ {+S^{-} \tilde{S}^{-} \int _{0}^{\frac{\pi }{2} }P\left(x\right)\cos K\left(x\right)dx } \\ {+\int _{0}^{\frac{\pi }{2} }P\left(x\right)\left(\int _{0}^{x}U_{c} \left(x,t\right)\cos \left(2\lambda t-K\left(t\right)\right)dt \right) dx} \\ {-\int _{0}^{\frac{\pi }{2} }P\left(x\right)\left(\int _{0}^{x}U_{s} \left(x,t\right)\sin \left(2\lambda t-K\left(t\right)\right)dt \right) dx.} \end{array}$\\
	where $\xi ^{\pm } \left(x\right)=\pm \alpha x\mp \alpha a_{1} +a_{1} $ , $k^{\pm } \left(x\right)=\mu ^{+} \left(a_{2} \right)\pm \beta x\mp \beta a_{2} $, \\
$s^{\pm } \left(x\right)=\mu ^{-} \left(a_{2} \right)\pm \beta x\mp \beta a_{2} ,\beta _{1} ^{\mp } =\frac{1}{2} \left(\alpha _{1} \mp \frac{\beta _{1} }{\alpha } \right) , \beta _{2} ^{\mp } =\frac{1}{2} \left(\alpha _{2} \mp \frac{\alpha \beta _{2} }{\beta } \right) .$ \\
$\begin{array}{l} {2U_{1} \left(\lambda \right)=\frac{S^{+} \tilde{S}^{+} }{2} \int _{0}^{\frac{\pi }{2} }P\left(t\right)e^{-i\left(K\left(t\right)\right)} e^{i\left(2\lambda \xi ^{+} \left(t\right)\right)} dt +\frac{S^{+} \tilde{S}^{+} }{2} \int _{0}^{\frac{\pi }{2} }P\left(t\right)e^{i\left(K\left(t\right)\right)} e^{-i\left(2\lambda \xi ^{+} \left(t\right)\right)} dt } \\ {+\frac{S^{+} \tilde{S}^{-} }{2} \int _{0}^{\frac{\pi }{2} }P\left(t\right)e^{-i\left(L\left(t\right)\right)} e^{i\left(2\lambda a_{1} t\right)} dt +\frac{S^{+} \tilde{S}^{-} }{2} \int _{0}^{\frac{\pi }{2} }P\left(t\right)e^{i\left(L\left(t\right)\right)} e^{-i\left(2\lambda a_{1} t\right)} dt } \\ {+\frac{S^{+} \tilde{S}^{-} }{2} \int _{0}^{\frac{\pi }{2} }P\left(t\right)e^{-i\left(K\left(t\right)\right)} e^{i\left(2\lambda \alpha \left(t-a_{1} \right)\right)} dt ++\frac{S^{+} \tilde{S}^{-} }{2} \int _{0}^{\frac{\pi }{2} }P\left(t\right)e^{i\left(K\left(t\right)\right)} e^{-i\left(2\lambda \alpha \left(t-a_{1} \right)\right)} dt } \\ {+\frac{S^{-} \tilde{S}^{+} }{2} \int _{0}^{\frac{\pi }{2} }P\left(t\right)e^{i\left(L\left(t\right)\right)} e^{i\left(2\lambda a_{1} t\right)} dt +\frac{S^{-} \tilde{S}^{+} }{2} \int _{0}^{\frac{\pi }{2} }P\left(t\right)e^{-i\left(L\left(t\right)\right)} e^{i\left(2\lambda a_{1} t\right)} dt } \\ {+\frac{S^{-} \tilde{S}^{+} }{2} \int _{0}^{\frac{\pi }{2} }P\left(t\right)e^{i\left(K\left(t\right)\right)} e^{i\left(2\lambda \alpha \left(t-a_{1} \right)\right)} dt +\frac{S^{-} \tilde{S}^{+} }{2} \int _{0}^{\frac{\pi }{2} }P\left(t\right)e^{-i\left(K\left(t\right)\right)} e^{i\left(2\lambda \alpha \left(t-a_{1} \right)\right)} dt } \\ {+\frac{S^{-} \tilde{S}^{-} }{2} \int _{0}^{\frac{\pi }{2} }P\left(t\right)e^{i\left(L\left(t\right)\right)} e^{-i\left(2\lambda \xi ^{-} \left(t\right)\right)} dt +\frac{S^{-} \tilde{S}^{-} }{2} \int _{0}^{\frac{\pi }{2} }P\left(t\right)e^{-i\left(L\left(t\right)\right)} e^{i\left(2\lambda \xi ^{-} \left(t\right)\right)} dt } \\ {+S^{+} \tilde{S}^{+} \int _{0}^{\frac{\pi }{2} }P\left(x\right)\cos L\left(x\right) dx+S^{-} \tilde{S}^{-} \int _{0}^{\frac{\pi }{2} }P\left(x\right)\cos K\left(x\right)dx } \\ {+\int _{0}^{\frac{\pi }{2} }P\left(x\right)\left(\int _{0}^{x}U_{c} \left(x,t\right)\cos \left(2\lambda t-K\left(t\right)\right)dt \right) dx} \\ {-\int _{0}^{\frac{\pi }{2} }P\left(x\right)\left(\int _{0}^{x}U_{s} \left(x,t\right)\sin \left(2\lambda t-K\left(t\right)\right)dt \right) dx} \end{array}$
	if necessary operations are performed and integrals are calculated\\
	
$\begin{array}{l} {2U_{1} \left(\lambda \right)=\frac{S^{+} \tilde{S}^{+} }{2} \left[\frac{T_{1} \left({\raise0.7ex\hbox{$ \pi  $}\!\mathord{\left/ {\vphantom {\pi  2}} \right. \kern-\nulldelimiterspace}\!\lower0.7ex\hbox{$ 2 $}} \right)}{2i\lambda \alpha } e^{i\left(2\lambda \xi ^{+} \left(\frac{\pi }{2} \right)\right)} -\frac{T_{1} \left(0\right)}{2i\lambda \alpha } e^{2i\lambda \left(\alpha a_{1} +a_{1} \right)} -\frac{1}{2i\lambda \alpha } \int _{0}^{\frac{\pi }{2} }T_{1} ^{{'} } \left(t\right)e^{i\left(2\lambda \xi ^{+} \left(t\right)\right)} dt \right]} \\ {+\frac{S^{+} \tilde{S}^{+} }{2} \left[-\frac{T_{2} \left({\raise0.7ex\hbox{$ \pi  $}\!\mathord{\left/ {\vphantom {\pi  2}} \right. \kern-\nulldelimiterspace}\!\lower0.7ex\hbox{$ 2 $}} \right)}{2i\lambda \alpha } e^{-i\left(2\lambda \xi ^{+} \left(\frac{\pi }{2} \right)\right)} +\frac{T_{2} \left(0\right)}{2i\lambda \alpha } e^{-2i\lambda \left(\alpha a_{1} +a_{1} \right)} +\frac{1}{2i\lambda \alpha } \int _{0}^{\frac{\pi }{2} }T_{2} ^{{'} } \left(t\right)e^{-i\left(2\lambda \xi ^{+} \left(t\right)\right)} dt \right]} \\ {+\frac{S^{+} \tilde{S}^{-} }{2} \left[\frac{T_{3} \left({\raise0.7ex\hbox{$ \pi  $}\!\mathord{\left/ {\vphantom {\pi  2}} \right. \kern-\nulldelimiterspace}\!\lower0.7ex\hbox{$ 2 $}} \right)}{2i\lambda \alpha } e^{i\lambda a_{1} } -\frac{T_{3} \left(0\right)}{2i\lambda \alpha } -\frac{1}{2i\lambda \alpha } \int _{0}^{\frac{\pi }{2} }T_{3} ^{{'} } \left(t\right)e^{2ia_{1} t} dt \right]} \\ {+\frac{S^{+} \tilde{S}^{-} }{2} \left[-\frac{T_{4} \left({\raise0.7ex\hbox{$ \pi  $}\!\mathord{\left/ {\vphantom {\pi  2}} \right. \kern-\nulldelimiterspace}\!\lower0.7ex\hbox{$ 2 $}} \right)}{2i\lambda \alpha } e^{i\lambda a_{1} } +\frac{T_{4} \left(0\right)}{2i\lambda \alpha } +\frac{1}{2i\lambda \alpha } \int _{0}^{\frac{\pi }{2} }T_{4} ^{{'} } \left(t\right)e^{-2ia_{1} t} dt \right]} \\ {+\frac{S^{+} S^{-} }{2} \left[\frac{T_{1} \left({\raise0.7ex\hbox{$ \pi  $}\!\mathord{\left/ {\vphantom {\pi  2}} \right. \kern-\nulldelimiterspace}\!\lower0.7ex\hbox{$ 2 $}} \right)}{2i\lambda \alpha } e^{2i\lambda \alpha \left(\frac{\pi }{2} -a_{1} \right)} -\frac{T_{1} \left(0\right)}{2i\lambda \alpha } e^{-2i\lambda \alpha a_{1} } -\frac{1}{2i\lambda \alpha } \int _{0}^{\frac{\pi }{2} }T_{1} ^{{'} } \left(t\right)e^{2i\lambda \alpha \left(t-a_{1} \right)} dt \right]} \\ {+\frac{S^{+} S^{-} }{2} \left[-\frac{T_{2} \left({\raise0.7ex\hbox{$ \pi  $}\!\mathord{\left/ {\vphantom {\pi  2}} \right. \kern-\nulldelimiterspace}\!\lower0.7ex\hbox{$ 2 $}} \right)}{2i\lambda \alpha } e^{-2i\lambda \alpha \left(\frac{\pi }{2} -a_{1} \right)} +\frac{T_{2} \left(0\right)}{2i\lambda \alpha } e^{2i\lambda \alpha a_{1} } +\frac{1}{2i\lambda \alpha } \int _{0}^{\frac{\pi }{2} }T_{2} ^{{'} } \left(t\right)e^{-2i\lambda \alpha \left(t-a_{1} \right)} dt \right]} \\ {+\frac{S^{-} \tilde{S}^{+} }{2} \left[-\frac{T_{3} \left({\raise0.7ex\hbox{$ \pi  $}\!\mathord{\left/ {\vphantom {\pi  2}} \right. \kern-\nulldelimiterspace}\!\lower0.7ex\hbox{$ 2 $}} \right)}{2i\lambda \alpha } e^{-i\lambda a_{1} \pi } +\frac{T_{3} \left(0\right)}{2i\lambda \alpha } +\frac{1}{2i\lambda \alpha } \int _{0}^{\frac{\pi }{2} }T_{3} ^{{'} } \left(t\right)e^{-2ia_{1} t} dt \right]} \\ {+\frac{S^{-} \tilde{S}^{+} }{2} \left[\frac{T_{4} \left({\raise0.7ex\hbox{$ \pi  $}\!\mathord{\left/ {\vphantom {\pi  2}} \right. \kern-\nulldelimiterspace}\!\lower0.7ex\hbox{$ 2 $}} \right)}{2i\lambda \alpha } e^{i\lambda a_{1} \pi } -\frac{T_{4} \left(0\right)}{2i\lambda \alpha } -\frac{1}{2i\lambda \alpha } \int _{0}^{\frac{\pi }{2} }T_{4} ^{{'} } \left(t\right)e^{2ia_{1} t} dt \right]} \\ {+\frac{S^{-} \tilde{S}^{+} }{2} \left[-\frac{T_{1} \left({\raise0.7ex\hbox{$ \pi  $}\!\mathord{\left/ {\vphantom {\pi  2}} \right. \kern-\nulldelimiterspace}\!\lower0.7ex\hbox{$ 2 $}} \right)}{2i\lambda \alpha } e^{-2i\lambda \alpha \left(\frac{\pi }{2} -a_{1} \right)} +\frac{T_{1} \left(0\right)}{2i\lambda \alpha } e^{2i\lambda \alpha a_{1} } +\frac{1}{2i\lambda \alpha } \int _{0}^{\frac{\pi }{2} }T_{1} ^{{'} } \left(t\right)e^{-2i\lambda \alpha \left(t-a_{1} \right)} dt \right]} \\ {+\frac{S^{-} \tilde{S}^{+} }{2} \left[\frac{T_{2} \left({\raise0.7ex\hbox{$ \pi  $}\!\mathord{\left/ {\vphantom {\pi  2}} \right. \kern-\nulldelimiterspace}\!\lower0.7ex\hbox{$ 2 $}} \right)}{2i\lambda \alpha } e^{2i\lambda \alpha \left(\frac{\pi }{2} -a_{1} \right)} -\frac{T_{2} \left(0\right)}{2i\lambda \alpha } e^{-2i\lambda \alpha a_{1} } -\frac{1}{2i\lambda \alpha } \int _{0}^{\frac{\pi }{2} }T_{2} ^{{'} } \left(t\right)e^{2i\lambda \alpha \left(t-a_{1} \right)} dt \right]} \\ {+\frac{S^{-} \tilde{S}^{-} }{2} \left[-\frac{T_{4} \left({\raise0.7ex\hbox{$ \pi  $}\!\mathord{\left/ {\vphantom {\pi  2}} \right. \kern-\nulldelimiterspace}\!\lower0.7ex\hbox{$ 2 $}} \right)}{2i\lambda \alpha } e^{i\left(2\lambda \xi ^{-} \left(\frac{\pi }{2} \right)\right)} +\frac{T_{4} \left(0\right)}{2i\lambda \alpha } e^{2i\lambda \left(\alpha a_{1} +a_{1} \right)} +\frac{1}{2i\lambda \alpha } \int _{0}^{\frac{\pi }{2} }T_{4} ^{{'} } \left(t\right)e^{i\left(2\lambda \xi ^{-} \left(t\right)\right)} dt \right]} \\ {+\frac{S^{-} \tilde{S}^{-} }{2} \left[\frac{T_{3} \left({\raise0.7ex\hbox{$ \pi  $}\!\mathord{\left/ {\vphantom {\pi  2}} \right. \kern-\nulldelimiterspace}\!\lower0.7ex\hbox{$ 2 $}} \right)}{2i\lambda \alpha } e^{-i\left(2\lambda \xi ^{-} \left(\frac{\pi }{2} \right)\right)} -\frac{T_{3} \left(0\right)}{2i\lambda \alpha } e^{2i\lambda \left(\alpha a_{1} -a_{1} \right)} -\frac{1}{2i\lambda \alpha } \int _{0}^{\frac{\pi }{2} }T_{3} ^{{'} } \left(t\right)e^{-i\left(2\lambda \xi ^{+} \left(t\right)\right)} dt \right]} \\ {+S^{+} \tilde{S}^{+} \int _{0}^{\frac{\pi }{2} }P\left(x\right)\cos L\left(x\right) dx+S^{-} \tilde{S}^{-} \int _{0}^{\frac{\pi }{2} }P\left(x\right)\cos K\left(x\right)dx } \\ {+\left[\frac{T_{5} \left({\raise0.7ex\hbox{$ \pi  $}\!\mathord{\left/ {\vphantom {\pi  2}} \right. \kern-\nulldelimiterspace}\!\lower0.7ex\hbox{$ 2 $}} \right)}{2i\lambda } e^{i\pi \lambda } -\frac{T_{5} \left(0\right)}{2i\lambda } -\frac{1}{2i\lambda } \int _{0}^{\frac{\pi }{2} }T'_{1} \left(t\right)e^{2i\lambda t} dt \right]} \\ {+\left[-\frac{T_{6} \left({\raise0.7ex\hbox{$ \pi  $}\!\mathord{\left/ {\vphantom {\pi  2}} \right. \kern-\nulldelimiterspace}\!\lower0.7ex\hbox{$ 2 $}} \right)}{2i\lambda } e^{-i\pi \lambda } +\frac{T_{6} \left(0\right)}{2i\lambda } +\frac{1}{2i\lambda } \int _{0}^{\frac{\pi }{2} }T'_{6} \left(t\right)e^{-2i\lambda t} dt \right]} \end{array}$

	\noindent where $T_{1} \left(t\right)=P\left(t\right)e^{-i\left(K\left(t\right)\right)} $, $T_{2} \left(t\right)=P\left(t\right)e^{i\left(K\left(t\right)\right)} $, $T_{3} \left(t\right)=P\left(t\right)e^{-i\left(L\left(t\right)\right)} $, $T_{4} \left(t\right)=P\left(t\right)e^{i\left(L\left(t\right)\right)} $,
$P_{1} \left(t\right)=\int _{t}^{\frac{\pi }{2} }P\left(x\right)U_{c} \left(x,t\right) dx, P_{2} \left(t\right)=\int _{t}^{\frac{\pi }{2} }P\left(x\right)U_{s} \left(x,t\right) dx$\\
$ , T_{5} \left(t\right)=\frac{P_{1} \left(t\right)+iP_{2} \left(t\right)}{2} e^{-iK\left(t\right)} ,T_{6} \left(t\right)=\frac{P_{1} \left(t\right)-iP_{2} \left(t\right)}{2} e^{iK\left(t\right)} $
	
	\noindent By the Riemann-Lebesgue lemma $\int _{0}^{\frac{\pi }{2} }P\left(x\right)\cos L\left(x\right) dx=0\, \, ,\, \,\\ \int _{0}^{\frac{\pi }{2} }P\left(x\right)\cos K\left(x\right)dx =0$ and $P\left(\frac{\pi }{2} \right)=0$ for $\lambda \to \infty $.\\ Thus,
\begin{equation} \label{24)} 
	\begin{array}{l} {2U_{1} \left(\lambda \right)=-\frac{S^{+} \tilde{S}^{+} }{4i\lambda \alpha } \int _{0}^{\frac{\pi }{2} }T_{1} ^{{'} } \left(t\right)e^{i\left(2\lambda \xi ^{+} \left(t\right)\right)} dt +\frac{S^{+} \tilde{S}^{+} }{4i\lambda \alpha } \int _{0}^{\frac{\pi }{2} }T_{2} ^{{'} } \left(t\right)e^{-i\left(2\lambda \xi ^{+} \left(t\right)\right)} dt } \\ {-\frac{S^{+} \tilde{S}^{-} }{4i\lambda \alpha } \int _{0}^{\frac{\pi }{2} }T_{3} ^{{'} } \left(t\right)e^{2ia_{1} t} dt +\frac{S^{+} \tilde{S}^{-} }{4i\lambda \alpha } \int _{0}^{\frac{\pi }{2} }T_{4} ^{{'} } \left(t\right)e^{-2ia_{1} t} dt } \\ {-\frac{S^{+} S^{-} }{4i\lambda \alpha } \int _{0}^{\frac{\pi }{2} }T_{1} ^{{'} } \left(t\right)e^{2i\lambda \alpha \left(t-a_{1} \right)} dt +\frac{S^{+} S^{-} }{4i\lambda \alpha } \int _{0}^{\frac{\pi }{2} }T_{2} ^{{'} } \left(t\right)e^{-2i\lambda \alpha \left(t-a_{1} \right)} dt } \\ {+\frac{S^{-} \tilde{S}^{+} }{4i\lambda \alpha } \int _{0}^{\frac{\pi }{2} }T_{3} ^{{'} } \left(t\right)e^{-2ia_{1} t} dt -\frac{S^{-} \tilde{S}^{+} }{4i\lambda \alpha } \int _{0}^{\frac{\pi }{2} }T_{4} ^{{'} } \left(t\right)e^{2ia_{1} t} dt } \\ {+\frac{S^{-} \tilde{S}^{+} }{4i\lambda \alpha } \int _{0}^{\frac{\pi }{2} }T_{1} ^{{'} } \left(t\right)e^{-2i\lambda \alpha \left(t-a_{1} \right)} dt +\frac{S^{-} \tilde{S}^{+} }{4i\lambda \alpha } \int _{0}^{\frac{\pi }{2} }T_{2} ^{{'} } \left(t\right)e^{2i\lambda \alpha \left(t-a_{1} \right)} dt } \\ {+\frac{S^{-} \tilde{S}^{-} }{4i\lambda \alpha } \int _{0}^{\frac{\pi }{2} }T_{4} ^{{'} } \left(t\right)e^{i\left(2\lambda \xi ^{-} \left(t\right)\right)} dt -\frac{S^{-} \tilde{S}^{-} }{4i\lambda \alpha } \int _{0}^{\frac{\pi }{2} }T_{3} ^{{'} } \left(t\right)e^{-i\left(2\lambda \xi ^{-} \left(t\right)\right)} dt } \\ {+\frac{i}{2\lambda } \int _{0}^{\frac{\pi }{2} }T'_{5} \left(t\right)e^{2i\lambda t} dt -\frac{i}{2\lambda } \int _{0}^{\frac{\pi }{2} }T'_{6} \left(t\right)e^{-2i\lambda t} dt } \end{array} 
\end{equation} 
$\begin{array}{l} {2U_{2} \left(\lambda \right)=S^{+} \tilde{S}^{+} \int _{0}^{\frac{\pi }{2} }Q\left(x\right)\left(\frac{e^{i\left(2\lambda \xi ^{+} \left(x\right)-K\left(x\right)\right)} +e^{-i\left(2\lambda \xi ^{+} \left(x\right)-K\left(x\right)\right)} }{2} \right)dx } \\ {+S^{+} \tilde{S}^{-} \int _{0}^{\frac{\pi }{2} }Q\left(x\right)\left(\frac{e^{i\left(2\lambda a_{1} t-L\left(x\right)\right)} +e^{-i\left(2\lambda a_{1} t-L\left(x\right)\right)} }{2} \right)dx } \\ {+S^{+} \tilde{S}^{-} \int _{0}^{\frac{\pi }{2} }Q\left(x\right)\left(\frac{e^{i\left(2\lambda \alpha \left(x-a_{1} \right)-K\left(x\right)\right)} +e^{-i\left(2\lambda \alpha \left(x-a_{1} \right)-K\left(x\right)\right)} }{2} \right)dx } \\ {+S^{-} \tilde{S}^{+} \int _{0}^{\frac{\pi }{2} }Q\left(x\right)\left(\frac{e^{i\left(2\lambda a_{1} t+L\left(x\right)\right)} +e^{-i\left(2\lambda a_{1} t+L\left(x\right)\right)} }{2} \right)dx } \\ {+S^{-} \tilde{S}^{+} \int _{0}^{\frac{\pi }{2} }Q\left(x\right)\left(\frac{e^{i\left(2\lambda \alpha \left(x-a_{1} \right)+K\left(x\right)\right)} +e^{-i\left(2\lambda \alpha \left(x-a_{1} \right)+K\left(x\right)\right)} }{2} \right)dx } \\ {+S^{-} \tilde{S}^{-} \int _{0}^{\frac{\pi }{2} }Q\left(x\right)\left(\frac{e^{i\left(2\lambda \xi ^{-} \left(x\right)+L\left(x\right)\right)} +e^{-i\left(2\lambda \xi ^{-} \left(x\right)+L\left(x\right)\right)} }{2} \right)dx } \\ {+S^{+} \tilde{S}^{+} \int _{0}^{\frac{\pi }{2} }Q\left(x\right)\cos L\left(x\right) dx+S^{-} \tilde{S}^{-} \int _{0}^{\frac{\pi }{2} }Q\left(x\right)\cos K\left(x\right)dx } \\ {+\int _{0}^{\frac{\pi }{2} }Q\left(x\right)\left(\int _{0}^{x}U_{c} \left(x,t\right)\cos \left(2\lambda t-K\left(t\right)\right)dt \right) dx} \\ {-\int _{0}^{\frac{\pi }{2} }Q\left(x\right)\left(\int _{0}^{x}U_{s} \left(x,t\right)\sin \left(2\lambda t-K\left(t\right)\right)dt \right) dx} \end{array}$

	\noindent where $R_{1} \left(t\right)=Q\left(t\right)e^{-i\left(K\left(t\right)\right)} $, $R_{2} \left(t\right)=Q\left(t\right)e^{i\left(K\left(t\right)\right)} $, $R_{3} \left(t\right)=Q\left(t\right)e^{-i\left(L\left(t\right)\right)} $,
$R_{4} \left(t\right)=Q\left(t\right)e^{i\left(L\left(t\right)\right)} ,Q_{1} \left(t\right)=\int _{t}^{\frac{\pi }{2} }P\left(x\right)U_{c} \left(x,t\right) dx,$\\
$ Q_{2} \left(t\right)=\int _{t}^{\frac{\pi }{2} }P\left(x\right)U_{s} \left(x,t\right) dx, R_{5} \left(t\right)=\frac{Q_{1} \left(t\right)+iQ_{2} \left(t\right)}{2} e^{-iK\left(t\right)} , $
$R_{6} \left(t\right)=\frac{Q_{1} \left(t\right)-iQ_{2} \left(t\right)}{2} e^{iK\left(t\right)} $\\
	By the Riemann-Lebesgue lemma\\ $\int _{0}^{\frac{\pi }{2} }Q\left(x\right)\cos L\left(x\right) dx=0\, \, ,\, \, \int _{0}^{\frac{\pi }{2} }Q\left(x\right)\cos K\left(x\right)dx =0$. Thus,
	\begin{equation} \label{25)} 
	\begin{array}{l} {2U_{2} \left(\lambda \right)=\frac{S^{+} \tilde{S}^{+} }{2} \int _{0}^{\frac{\pi }{2} }R_{1} \left(t\right)e^{i\left(2\lambda \xi ^{+} \left(t\right)\right)} dt +\frac{S^{+} \tilde{S}^{+} }{2} \int _{0}^{\frac{\pi }{2} }R_{2} \left(t\right)e^{-i\left(2\lambda \xi ^{+} \left(t\right)\right)} dt } \\ {+\frac{S^{+} \tilde{S}^{-} }{2} \int _{0}^{\frac{\pi }{2} }R_{3} \left(t\right)e^{2ia_{1} t} dt +\frac{S^{+} \tilde{S}^{-} }{2} \int _{0}^{\frac{\pi }{2} }R_{4} \left(t\right)e^{-2ia_{1} t} dt } \\ {+\frac{S^{+} S^{-} }{2} \int _{0}^{\frac{\pi }{2} }R_{1} \left(t\right)e^{2i\lambda \alpha \left(t-a_{1} \right)} dt +\frac{S^{+} S^{-} }{2} \int _{0}^{\frac{\pi }{2} }R_{2} \left(t\right)e^{-2i\lambda \alpha \left(t-a_{1} \right)} dt } \\ {+\frac{S^{-} \tilde{S}^{+} }{2} \int _{0}^{\frac{\pi }{2} }R_{3} \left(t\right)e^{-2ia_{1} t} dt +\frac{S^{-} \tilde{S}^{+} }{2} \int _{0}^{\frac{\pi }{2} }R_{4} \left(t\right)e^{2ia_{1} t} dt } \\ {+\frac{S^{-} \tilde{S}^{+} }{2} \int _{0}^{\frac{\pi }{2} }R_{1} \left(t\right)e^{-2i\lambda \alpha \left(t-a_{1} \right)} dt +\frac{S^{-} \tilde{S}^{+} }{2} \int _{0}^{\frac{\pi }{2} }R_{2} \left(t\right)e^{2i\lambda \alpha \left(t-a_{1} \right)} dt } \\ {+\frac{S^{-} \tilde{S}^{-} }{2} \int _{0}^{\frac{\pi }{2} }R_{4} \left(t\right)e^{i\left(2\lambda \xi ^{-} \left(t\right)\right)} dt +\frac{S^{-} \tilde{S}^{-} }{2} \int _{0}^{\frac{\pi }{2} }R_{3} \left(t\right)e^{-i\left(2\lambda \xi ^{-} \left(t\right)\right)} dt } \\ {+\frac{i}{2\lambda } \int _{0}^{\frac{\pi }{2} }R_{5} \left(t\right)e^{2i\lambda t} dt +\frac{i}{2\lambda } \int _{0}^{\frac{\pi }{2} }R_{6} \left(t\right)e^{-2i\lambda t} dt } \end{array} 
	\end{equation} 
	
	\begin{equation} \label{26)} 
	2\lambda U_{1} \left(\lambda \right)+U_{2} \left(\lambda \right)=0. 
	\end{equation} 
	If $\left(2.10\right)$ and $\left(2.11\right)$ are substituted into $\left(2.12\right)$, we get\\
$\begin{array}{l} {\frac{S^{+} \tilde{S}^{+} }{2\alpha } \int _{0}^{\frac{\pi }{2} }\left(R_{1} \left(t\right)+iT'_{1} \left(t\right)\right)e^{i\left(2\lambda \xi ^{+} \left(t\right)\right)} dt +\frac{S^{+} \tilde{S}^{+} }{2\alpha } \int _{0}^{\frac{\pi }{2} }\left(R_{2} \left(t\right)-iT'_{2} \left(t\right)\right)e^{-i\left(2\lambda \xi ^{+} \left(t\right)\right)} dt } \\ {+\frac{S^{+} \tilde{S}^{-} }{2\alpha } \int _{0}^{\frac{\pi }{2} }\left(R_{3} \left(t\right)+iT'_{3} \left(t\right)\right)e^{2ia_{1} t} dt +\frac{S^{+} \tilde{S}^{-} }{2\alpha } \int _{0}^{\frac{\pi }{2} }\left(R_{4} \left(t\right)-iT'_{4} \left(t\right)\right)e^{-2ia_{1} t} dt } \\ {+\frac{S^{+} S^{-} }{2\alpha } \int _{0}^{\frac{\pi }{2} }\left(R_{1} \left(t\right)+iT'_{1} \left(t\right)\right)e^{2i\lambda \alpha \left(t-a_{1} \right)} dt +\frac{S^{+} S^{-} }{2\alpha } \int _{0}^{\frac{\pi }{2} }\left(R_{2} \left(t\right)-iT'_{2} \left(t\right)\right)e^{-2i\lambda \alpha \left(t-a_{1} \right)} dt } \\ {+\frac{S^{-} \tilde{S}^{+} }{2\alpha } \int _{0}^{\frac{\pi }{2} }\left(R_{4} \left(t\right)+iT'_{4} \left(t\right)\right)e^{-2ia_{1} t} dt +\frac{S^{-} \tilde{S}^{+} }{2\alpha } \int _{0}^{\frac{\pi }{2} }\left(R_{3} \left(t\right)-iT'_{3} \left(t\right)\right)e^{2ia_{1} t} dt } \\ {+\frac{S^{-} \tilde{S}^{+} }{2} \int _{0}^{\frac{\pi }{2} }\left(R_{2} \left(t\right)+iT'_{2} \left(t\right)\right)e^{2i\lambda \alpha \left(t-a_{1} \right)} dt +\frac{S^{-} \tilde{S}^{+} }{2} \int _{0}^{\frac{\pi }{2} }\left(R_{1} \left(t\right)-iT'_{1} \left(t\right)\right)e^{-2i\lambda \alpha \left(t-a_{1} \right)} dt } \\ {+\frac{S^{-} \tilde{S}^{-} }{2} \int _{0}^{\frac{\pi }{2} }\left(R_{4} \left(t\right)-iT'_{4} \left(t\right)\right)e^{i\left(2\lambda \xi ^{-} \left(t\right)\right)} dt +\frac{S^{-} \tilde{S}^{-} }{2} \int _{0}^{\frac{\pi }{2} }\left(R_{3} \left(t\right)+iT'_{3} \left(t\right)\right)e^{-i\left(2\lambda \xi ^{-} \left(t\right)\right)} dt } \\ {+\int _{0}^{\frac{\pi }{2} }\left(R_{5} \left(t\right)+iT'_{5} \left(t\right)\right)e^{2i\lambda t} dt +\int _{0}^{\frac{\pi }{2} }\left(R_{6} \left(t\right)-iT'_{6} \left(t\right)\right)e^{-2i\lambda t} dt =0} \end{array}$\\ 
	Since the systems $\left\{e^{\pm 2i\lambda \xi ^{+} \left(t\right)} :\, \, \lambda \in {\rm R}\right\}$ , $\left\{e^{\pm 2i\lambda a_{1} t} :\, \, \lambda \in {\rm R}\right\}$, $\left\{e^{\pm 2i\lambda \alpha \left(t-a_{1} \right)} :\, \, \lambda \in {\rm R}\right\}$ and $\left\{e^{\pm 2i\lambda t} :\, \, \lambda \in {\rm R}\right\}$ are entire in $L_{2} \left(-\frac{\pi }{2} ,\frac{\pi }{2} \right)$, it follows
	\[\begin{array}{l} {R_{1} \left(t\right)+iT'_{1} \left(t\right)=0\, \, ,\, \, R_{2} \left(t\right)-iT'_{2} \left(t\right)=0\, \, ,\, \, R_{3} \left(t\right)+iT'_{3} \left(t\right)=0} \\ {R_{4} \left(t\right)-iT'_{4} \left(t\right)=0\, \, ,\, \, R_{1} \left(t\right)+iT'_{1} \left(t\right)=0\, \, ,\, \, R_{2} \left(t\right)-iT'_{2} \left(t\right)=0} \\ {R_{4} \left(t\right)+iT'_{4} \left(t\right)=0\, \, ,\, \, R_{3} \left(t\right)-iT'_{3} \left(t\right)=0\, \, ,\, \, R_{2} \left(t\right)+iT'_{2} \left(t\right)=0} \\ {R_{1} \left(t\right)-iT'_{1} \left(t\right)=0\, \, ,\, \, R_{4} \left(t\right)-iT'_{4} \left(t\right)=0\, \, ,\, \, \, R_{3} \left(t\right)+iT'_{3} \left(t\right)=0} \\ {R_{5} \left(t\right)+iT'_{5} \left(t\right)=0\, \, ,\, \, R_{6} \left(t\right)-iT'_{6} \left(t\right)=0} \end{array}\] 
	Then, we get the following system.
	\[\begin{array}{l} {R_{5} \left(t\right)+iT'_{5} \left(t\right)=0\, \, } \\ {R_{6} \left(t\right)-iT'_{6} \left(t\right)=0} \end{array}\] 
	and hence,
	\[\left\{\begin{array}{l} {\left[Q_{1} \left(t\right)+P_{1} \left(t\right)K'\left(t\right)-P_{2} ^{{'} } \left(t\right)\right]+i\left[Q_{2} \left(t\right)+P_{2} \left(t\right)K'\left(t\right)+P_{1} ^{{'} } \left(t\right)\right]=0} \\ {\left[Q_{1} \left(t\right)+P_{1} \left(t\right)K'\left(t\right)-P_{2} ^{{'} } \left(t\right)\right]-i\left[Q_{2} \left(t\right)+P_{2} \left(t\right)K'\left(t\right)+P_{1} ^{{'} } \left(t\right)\right]=0} \end{array}\right. \] 
	and hence,
	\[\left\{\begin{array}{l} {Q_{1} \left(t\right)+P_{1} \left(t\right)K'\left(t\right)-P_{2} ^{{'} } \left(t\right)=0} \\ {Q_{2} \left(t\right)+P_{2} \left(t\right)K'\left(t\right)+P_{1} ^{{'} } \left(t\right)=0} \end{array}\right. \] 
	\begin{equation} \label{27)} 
	\left\{\begin{array}{l} {P'\left(t\right)=U_{c} \left(t,t\right)P\left(t\right)} \\ {-\int _{t}^{\frac{\pi }{2} }U_{s} \left(x,t\right)Q\left(x\right)dx-\int _{t}^{\frac{\pi }{2} }\left(K'\left(t\right)U_{s} \left(x,t\right)+\frac{\partial H_{s} \left(x,t\right)}{\partial t} \right)P\left(x\right)dx  } \\ {} \\ {P\left(t\right)=-\int _{t}^{\frac{\pi }{2} }P'\left(x\right)dx } \\ {} \\ {Q\left(t\right)=-\left(K'\left(t\right)+U_{s} \left(t,t\right)\right)P\left(t\right)} \\ {-\int _{t}^{\frac{\pi }{2} }U_{c} \left(x,t\right)Q\left(x\right)dx-\int _{t}^{\frac{\pi }{2} }\left(K'\left(t\right)U_{c} \left(x,t\right)-\frac{\partial H_{s} \left(x,t\right)}{\partial t} \right)P\left(x\right)dx  } \end{array}\right.  
	\end{equation} 
	If we mark this
	\[S\left(t\right)=\left(Q\left(t\right),P\left(t\right),P'\left(t\right)\right)^{T} \] 
	and
	\[K\left(x,t\right)=\left(\begin{array}{ccc} {U_{c} \left(x,t\right)} & {K'\left(t\right)U_{c} \left(x,t\right)-\frac{\partial U_{s} \left(x,t\right)}{\partial t} } & {-\left(K'\left(t\right)+U_{s} \left(t,t\right)\right)} \\ {0} & {0} & {1} \\ {U_{s} \left(x,t\right)} & {K'\left(t\right)U_{s} \left(x,t\right)+\frac{\partial U_{s} \left(x,t\right)}{\partial t} } & {U_{c} \left(x,t\right)} \end{array}\right)\] 
	Equations $\left(2.13\right)$ can be reduced to a vector from
	\begin{equation} \label{28)}
S\left(t\right)+\int _{t}^{\frac{\pi }{2} }K\left(x,t\right)S\left(x\right)dx=0  
	\end{equation}
 for $0<t<\frac{\pi }{2} $.
	
 Since the equation $\left(2.14\right)$is a homogenous Volterra integral equations. Equation $\left(2.14\right)$ only has the trivial solution. Thus, we obtain 
	
	\noindent $S\left(t\right)=0$ for $0<t<\frac{\pi }{2} $.
	
	\noindent This gives us 
	
	\noindent $Q\left(t\right)=P\left(t\right)=0$ for $0<t<\frac{\pi }{2} $.
	
	\noindent Thus, we obtain $q\left(x\right)=\tilde{q}\left(x\right)$ and $p\left(x\right)=\tilde{p}\left(x\right)$ on $\left(0,\pi \right)$.
	The proof is comleted.
	
\end{proof}

\section*{Acknowledgement} 

Not applicable.
\bibliography{mmnsample}
\bibliographystyle{mmn}

\end{document}